\documentclass[a4paper]{scrartcl}

\usepackage{amsfonts,amsthm,amssymb,amsmath}
\usepackage[latin1]{inputenc}
\usepackage[T1]{fontenc}
\usepackage{graphicx}
\usepackage[inline]{enumitem}
\usepackage{hyperref}
\usepackage{tikz}
\usepackage{tkz-berge}

\newtheorem{thm}{Theorem}
\newtheorem{lem}[thm]{Lemma}

\newtheorem{cor}[thm]{Corollary}

\newtheorem{con}[thm]{Conjecture}

\theoremstyle{definition}

\newcommand{\Aut}{\operatorname{Aut}}

\title{Breaking graph symmetries by edge colourings}
\author{Florian Lehner}

\begin{document}
\maketitle

\begin{abstract}
The distinguishing index $D'(G)$ of a graph $G$ is the least number of colours needed in an edge colouring which is not preserved by any non-trivial automorphism. Broere and Pil\'sniak conjectured that if every non-trivial automorphism of a countable graph $G$ moves infinitely many edges, then $D'(G) \leq 2$. We prove this conjecture.
\end{abstract}

\section{Introduction}
A colouring of the vertices or edges of a graph $G$ is called distinguishing if the only automorphism which preserves it is the identity. Originally inspired by a recreational mathematics problem, Albertson and Collins~\cite{MR1394549} first introduced the notion formally in 1996. Despite (or maybe because of) its recreational origin, the concept quickly received a lot of attention leading to numerous papers on distinguishing colourings of graphs and other combinatorial structures.

One interesting line of research is the connection between motion (i.e.\ the minimal number of elements moved by a non-trivial automorphism) and the least number of colours needed in a distinguishing colouring. Intuitively, the more elements are moved by every non-trivial automorphism, the easier it should be to find a colouring with few colours which isn't preserved by any of them. Russel and Sundaram~\cite{MR1617449} were the first to make this intuition precise. They showed that if the motion of a finite graph is at least $2 \cdot \log_2 |\Aut G|$, then there is a distinguishing $2$-colouring. The same is true for infinite graphs whose automorphism group is finite. Tucker~\cite{zbMATH05902980} conjectured, that an analogous result holds for locally finite graphs with infinite automorphism group.
\begin{con}[Infinite motion conjecture~\cite{zbMATH05902980}]
\label{con:imc}
Let $G$ be a locally finite, connected graph and assume that every automorphism of $G$ moves infinitely many vertices. Then there is a distinguishing $2$-vertex colouring.
\end{con}
Note that if the motion of such a graph is infinite, then then it must be $\aleph_0$ and that $2^{\aleph_0}$ is a trivial upper bound for the size of the automorphism group.

While Tucker's conjecture is still wide open, there  are many partial results towards it, see~\cite{zbMATH06351222,MR2302543,zbMATH06405771,zbMATH06261159,growth,zbMATH06045720,MR2302536}. Broere and Pil\'sniak~\cite{zbMATH06428682} noticed that most of these partial results can be generalised to edge colourings. Consequently, they conjecured that an analogous statement to Conjecture~\ref{con:imc} should hold in the realm of edge colourings. In fact their conjecture for edge colourings is even stronger as it doesn't require the graph to be locally finite.
\begin{con}[Infinite edge motion conjecture~\cite{zbMATH06428682}]
\label{con:iemc}
Let $G$ be a countable, connected graph and assume that every automorphism of $G$ moves infinitely many edges. Then there is a distinguishing $2$-edge colouring.
\end{con}

The two conjectures are closely related. In~\cite{zbMATH03228459}, a generalisation of Whitney's theorem is proved, stating that for connected graphs on more than $4$ vertices (and in particular for infinite graphs) there is a natural group isomorphism between $\Aut G$ and $\Aut L(G)$, where $L(G)$ denotes the line graph of $G$. Hence, a distinguishing vertex colouring of $L(G)$ translates into a distinguishing edge colouring of $G$ and vice versa.  In particular, Conjecture~\ref{con:imc} implies the special case of Conjecture~\ref{con:iemc} where the graph is assumed to be locally finite. 

Furthermore, if the generalisation of Conjecture~\ref{con:imc} to countable graphs were true, then this would immediately imply Conjecture~\ref{con:iemc}. However, in ~\cite{zbMATH06502805} a counterexample for this generalisation is constructed, making it somewhat counterintuitive that Conjecture~\ref{con:iemc} holds in full generality.

Nevertheless, in the present paper we prove Conjecture~\ref{con:iemc}. We also attempt to give some intuition why this is not as surprising as it may seem at first glance. For this purpose, in Section \ref{sec:edgevertex} we compare distinguishing edge and vertex colourings. We show that if there is a distinguishing vertex colouring with $k$ colours, then there is a distinguishing edge colouring using at most $k+1$ colours. This is true for arbitrary graphs. One possible interpretation of this result is that finding a distinguishing edge colouring with few colours should generally be easier (or at least not harder) then finding such a vertex colouring and consequently that Conjecture~\ref{con:iemc} should be weaker than its vertex colouring counterpart.

\section{Notions and notations}
\label{sec:notions}

We will follow the terminology of~\cite{MR2159259} for all graph theoretical notions which are not explicitly defined. Let $G = (V,E)$ be a graph and let $\Aut G$ denote its automorphism group. A \emph{vertex colouring} of $G$ with colours in $C$ is a map $c \colon V \to C$. Analogously define an \emph{edge colouring}. We say that $\gamma \in \Aut G$ \emph{preserves} the (vertex or edge) colouring $c$ if $c \circ \gamma = c$. Two colourings $c$ and $d$ are called \emph{isomorphic}, if there is $\gamma \in \Aut G$ such that $c \circ \gamma = d$.

Call a colouring of $G$ \emph{distinguishing}, if the identity is the only automorphism which preserves it. The \emph{distinguishing number} of $G$, denoted by $D(G)$ is the least number of colours in a distinguishing vertex colouring. The \emph{distinguishing index} of $G$, denoted by $D'(G)$ is the analogous concept for edge colourings.

The \emph{motion} of a graph $G$ is the least number of vertices moved by a non-trivial automorphism of $G$. The \emph{edge motion} is the least number of edges moved by a non-trivial automorphism.

\section{Infinite motion and 2-distinguishability}
\label{sec:iemc}

In this section we prove Conjecture~\ref{con:iemc}. The following lemma will be useful.

\begin{lem}
\label{lem:infcomp}
Let $G$ be a graph with infinite edge motion and let $\gamma \in \Aut G$. Denote by $V_{\text{move}}$ the set of vertices of $G$ which are not fixed by $\gamma$. Let $C$ be the vertex set of a component of the subgraph of $G$ induced by $V_{\text{move}}$. If $C$ is finite, then it must contain a vertex of infinite degree.
\end{lem}

\begin{proof}
Assume for a contradiction that $C$ is finite and contains no vertex of infinite degree. Then $\gamma$ moves $C$ to some component $C'$ of $G[V_{\text{move}}]$. Denote by $\partial C$ the set of vertices outside of $C$ with a neighbour in $C$. Then each vertex of $\partial C$ is fixed by $\gamma$. If $C=C'$ we hence get the following automorphism:
\[
	\gamma'(v) = 
	\begin{cases}
	\gamma(v) & \text{if }v \in C,\\
	v & \text{if }v \notin C.
	\end{cases}
\]
Now $\gamma'$ only moves finitely many vertices all of which have finite degree. Hence it is an automorphism of $G$ with finite edge motion contradicting the fact that $G$ has infinite edge motion.

If $C \neq C'$ then define the following automorphism:
\[
	\gamma'(v) = 
	\begin{cases}
	\gamma(v) & \text{if }v \in C,\\
	\gamma^{-1}(v) & \text{if }v \in C',\\
	v & \text{otherwise}.
	\end{cases}
\]
Again this is an automorphism with finite edge motion contradicting the fact that $G$ has infinite edge motion.
\end{proof}

\begin{thm}
Every countable graph with infinite edge motion has $2^{\aleph_0}$ non-isomorphic distinguishing $2$-edge colourings.
\end{thm}

\begin{proof}
We first show that there is a distinguishing edge colouring by giving an explicit construction and then argue that within this construction we can make sufficiently many choices to obtain $2^{\aleph_0}$ non-isomorphic colourings.  

For the construction of the colouring we start with a colouring where all edges are coloured white and describe an inductive procedure to decide on edges whose colour will be changed to black. Since we change the colour of every edge at most once, we get a limit colouring which we will show to be distinguishing.

First we consider only edges incident to vertices of infinite degree. For this purpose choose an enumeration $(v_n^\infty)_{n \in \mathbb N}$ of these vertices and a strictly increasing sequence $(d_n)_{n \in \mathbb  N}$ of natural numbers. Note that $d_n \geq n$ because the sequence is strictly increasing.

Now inductively recolour edges incident to $v_n^{\infty}$ such that this vertex is incident to exactly $d_n$ black edges. We can do so without recolouring any edges incident to any vertex appearing earlier in the enumeration because there are at most $n-1$ such edges incident to $v_n^\infty$. In particular, since $d_n \geq n$ there can't be more than $d_n$ black edges incident to $v_n^\infty$ before step $n$.

With the colouring described above, no matter how we colour the remaining edges, every colour preserving automorphism must fix every vertex of infinite degree. This is because vertices of infinite degree must be mapped to vertices of infinite degree, and all of them have different degrees in the graph spanned by the black edges.

Now denote by $G_{\not\infty}$ the graph obtained from $G$ by deleting all vertices of infinite degree. Since all vertices of infinite degree must be fixed by every colour preserving automorphism it follows from Lemma~\ref{lem:infcomp} that no such automorphism of $G$ moves any vertex contained in a finite component of $G_{\not \infty}$.

Hence we only need to take care of automorphisms moving vertices in infinite components of $G_{\not \infty}$. For this purpose let $(C_k)_{k \in \mathbb N}$ be an enumeration of these components and let $l_n$ be a strictly increasing sequence of natural numbers with $l_1 > 1$. We will now recolour the edges of each $C_k$ in such a way that 
\begin{enumerate}[label=(\alph*)] 
\item \label{itm:disjointpaths} the subgraph of $C_k$ spanned by the black edges is a vertex disjoint union of paths of lengths $1, l_k, l_{k+1}, l_{k+2}, \ldots$, and
\item \label{itm:breaksetstab} there is no colour preserving automorphism of $G$ stabilising $C_k$ setwise but not pointwise. 
\end{enumerate}
Note that property~\ref{itm:disjointpaths} ensures that there is no automorphism which moves one infinite component to another because if $k_1 < k_2$ then $C_{k_1}$ contains a black path of length $l_{k_1}$ while $C_{k_2}$ doesn't. Property~\ref{itm:breaksetstab} makes sure that there is no automorphism mapping $C_k$ non-trivially to itself. Combined those two properties make sure that every colour preserving automorphism must fix every vertex in each $C_k$. Since we already established that each such automorphism must also fix all other vertices this means that we have found a distinguishing edge colouring.

We shall now construct a colouring of the edges of $C_k$ with properties~\ref{itm:disjointpaths} and \ref{itm:breaksetstab}. First we pick an edge $e$ of $C_k$  and colour it black. Define
\[
	S_i = \{v \in V(C_k) \mid d(v,e) = i\}
\]
where $d(v,e)$ denotes the minimal distance (measured in $C_k$) of $v$ to one of the two endvertices of $e$. Observe that $S_i$ must be finite for every $i$ since $C_k$ is locally finite.

Throughout the construction the edge $e$ will remain the only black edge which is not incident to any other black edge. Hence every colour preserving automorphism which maps $C_k$ onto itself must fix the edge $e$ and thus lie in the setwise stabiliser of every $S_i$. In particular, if such an automorphism acts non-trivially on $C_k$ then there is some $S_i$ on which it acts non-trivially.

Furthermore, throughout the construction it will be true that every colour preserving automorphism fixes each vertex in $V(C_k) \setminus S_0$ which is incident to a black edge. This fact will be useful to make sure that the paths we colour black are indeed disjoint.

In what follows we will only consider colour preserving automorphisms which stabilise $C_k$ setwise but not pointwise. We will denote the set of such automorphisms by $\Gamma$. Note that $\Gamma$ changes in each recolouring step. Furthermore, every $\gamma \in \Gamma$ must fix $V_\infty$ pointwise because it is assume to preserve the colouring constructed thus far. In particular we can without loss of generality assume that such an automorphism acts trivially outside of $C_k$.

Let $i \in \mathbb N$ be minimal with the property that there is $\gamma \in \Gamma$ which acts non-trivially on $S_i$. Choose $v \in S_i$ and $\gamma \in \Gamma$ such that $\gamma(v) \neq v$. Since all vertices moved by $\gamma$ have finite degree, Lemma~\ref{lem:infcomp} tells us that there must be a ray starting in $v$ which consists only of vertices which are moved by $\gamma$. Furthermore we can without loss of generality assume that the ray starts in $S_i$ and otherwise only contains vertices in $S_j$ for $j > i$. This can be achieved by moving to a suitable subray since $\bigcup_{j \leq i} S_j$ only contains finitely many vertices. Since every vertex of the ray is moved by $\gamma$ and $\gamma$ is colour preserving we infer that the ray contains no vertex which is incident to a black edge.

Now let $l_n$ be the smallest length in the sequence $l_k, l_{k+1},\ldots$ which has not been used yet. We colour an initial piece $P$ of length $l_n$ of our ray black and leave the rest of the colouring as it is.

Clearly $P$ is vertex disjoint from all other black paths constructed so far. Furthermore, since there is no other black path of length $l_n$ in $C_k$, each automorphism in $\Gamma$ must fix $P$ setwise (note that by recolouring $P$ we changed $\Gamma$). Finally every such automorphism must fix $P$ pointwise because only one endpoint of $P$ lies in the set $S_i$.

Since $|S_i|$ is an upper bound on the number of vertex disjoint paths starting at $S_i$ we end up with no $\gamma \in \Gamma$ acting non-trivially on $S_i$ after finitely many steps. Iterate the procedure with the next (larger) $i$. 

If after finitely many steps of this iteration we end up with $\Gamma = \{id\}$ then we put disjoint black paths of the remaining lengths anywhere in the (infinite) white part of $C_k$, otherwise continue inductively forever. This ensures that the colouring satisfies~\ref{itm:disjointpaths} which will be convenient when showing that there are continuum many non-isomorphic distinguishing colourings.

In the limit we get a colouring with infinitely many black paths of different lengths. An element of $\Gamma$ which preserves this limit colouring hence must preserve all of those paths setwise, in particular it must be a colour preserving automorphism of the colourings obtained in every single step of the construction. However, such an automorphism must act non-trivially on some $S_i$ which implies that there is some step for which it does not preserve the colouring. Hence the limit colouring satisfies properties~\ref{itm:disjointpaths} and \ref{itm:breaksetstab}. If we carry out this construction for every infinite component we thus get a distinguishing $2$-edge colouring.

It remains to show that we still have enough freedom in the construction to obtain $2^{\aleph_0}$ non-isomorphic such colourings. 

Firstly, if there are infinitely many vertices of infinite degree then each of the $2^{\aleph_0}$ choices for the sequence $(d_n)_{n \in \mathbb N}$ will deliver a colouring which is not isomorphic to any of the other colourings. The reason for this is, that vertices of infinite degree must be mapped to vertices of infinite degree while preserving the number of black edges incident to them. 

Secondly, if there is an infinite component of $G_{\not \infty}$ then each of the $2^{\aleph_0}$ choices for the sequence $(l_n)_{n \in \mathbb N}$ will give a colouring which is not isomorphic to any of the other colourings. Here the reason is, that black paths of length $l$ must be mapped to black paths of length $l$.

The only remaining case is that there are only finitely many vertices of infinite degree and all components of $G_{\not \infty}$ are finite. In this case we can colour the edges in the finite components any way we want (each choice giving a colouring which is not isomorphic to any of the others). Hence, if there are infinitely many such edges we are done. The only way this could fail is that all but finitely many of the components are singletons, so in particular $G_{\not \infty}$ must have infinitely many isolated vertices. But then there must be two such vertices which have the same neighbours in $G$ because there are only finitely many vertices of infinite degree. The transposition of two such vertices would be an automorphism of $G$ with finite edge motion, a contradiction to the assumption that $G$ has infinite edge motion.
\end{proof}

As a corollary to the above theorem we obtain another partial result towards Conjecture~\ref{con:imc}: we show that it is true for line graphs, the proof works even without the requirement of local finiteness.

\begin{cor}
Conjecture~\ref{con:imc} is true for line graphs.
\end{cor}

\begin{proof}
Let $L(G)$ be a countable, connected line graph of some graph $G$ with infinite motion. Then $G$ has only one component which contains edges, without loss of generality assume that $G$ is connected. By~\cite{zbMATH03228459} the automorphism groups of $G$ and $L(G)$ are isomorphic by means of the obvious map from $\Aut G$ to $\Aut L(G)$. This implies that $G$ has infinite edge motion and every distinguishing edge colouring of $G$ translates into a distinguishing vertex colouring of $L(G)$.
\end{proof}

\section{Edge colourings vs.\ vertex colourings}
\label{sec:edgevertex}

The purpose of this section is to compare distinguishing vertex and edge colourings. We will show how to construct from a distinguishing vertex colouring a distinguishing edge colouring using at most one more colour. The following construction will be our starting point.

Let $G$ be a graph and let $c\colon V \to C$ be a colouring of the vertex set of $G$ with colours in $C$. Without loss of generality assume that $C$ carries the additional structure of an Abelian group. Then we can obtain a colouring of the edge set by $e \mapsto c(u) + c(v)$ for $e = uv$. We will call such an edge colouring a \emph{canonical} edge colouring. We now derive some useful properties of canonical edge colourings.

%

\begin{lem}
\label{lem:nofixedpoint}
Let $G$ be a connected graph and let $c'$ be a canonical edge colouring of $G$ which comes from a distinguishing vertex colouring $c$. If there is a non-trivial automorphism preserving $c'$ then it doesn't preserve $c(v)$ for any vertex $v$. In particular such an automorphism can't have a fixed point in the vertex set.
\end{lem}

\begin{proof}
Let $\gamma$ be a non-trivial automorphism preserving $c'$. Since $\gamma$ preserves $c'$ we know that for every edge $e = uv$ it holds that 
\[
	c(u) + c(v) = c'(e) = c'(\gamma e) = c(\gamma u) + c(\gamma v).
\]

Now assume that there was a vertex $v_0$ such that $c(v_0) = c(\gamma v_0)$. Then for every neighbour $v$ of $v_0$ it we have
\[
	c(v_0) + c(v) =  c(\gamma v_0) + c(\gamma v) = c(v_0) + c(\gamma v),
\]
and thus $c(v) = c(\gamma v)$. 

By induction on the distance between $v$ and $v_0$ we obtain that $c(v) = c(\gamma v)$ for every vertex $v$ of $G$. Hence $\gamma$ preserves the colouring $c$ and thus it must be the identity.
\end{proof}

\begin{cor}
Let $G$ be a connected graph and let $c'$ be a canonical edge colouring corresponding to a distinguishing vertex colouring $c$ with colours in $C$. Let $\Gamma$ be the stabiliser of $c'$ in $\Aut G$.  Then $|\Gamma| \leq |C|$.
\end{cor}
\begin{proof}
By Lemma~\ref{lem:nofixedpoint} there is no vertex which is fixed by any $\gamma \in \Gamma$. Hence the size of $\Gamma$ is bounded from above by the size of the orbits on the vertex set. If the size of any orbit were $> |C|$, then there would be two different vertices with the same colour in this orbit. Hence we would have a non-trivial automorphism mapping some vertex to a vertex with the same colour, a contradiction.
\end{proof}

The above results show that a canonical edge colouring corresponding to a distinguishing vertex colouring is not far from being distinguishing. We now show how to modify it in order to obtain a distinguishing edge colouring using only one additional colour.
We note that a finite version of the following theorem has been proved (using an entirely different approach) in \cite{zbMATH06381902}. A proof for infinite graphs following essentially the same lines as the finite proof has been announced by Imrich et al.~\cite{distindex-compare}. It is also worth mentioning that the bound is known to be tight as there is a family of finite trees for which equality holds, see \cite{zbMATH06381902}.

\begin{thm}
\label{thm:differenceone}
Let $G$ be a connected graph, then $D'(G) \leq D(G)+1$.
\end{thm}

\begin{proof}
Let $c$ be a distinguishing $k$-colouring and let $c'$ be the corresponding canonical edge colouring. 

If there are two incident edges with the same colour, then take two such edges $e$ and $f$ and colour both of them with a new colour $x$. An automorphism which preserves the resulting colouring either fixes both edges or swaps them. In both cases it is easy to verify that such an automorphism preserves $c'$. But since $e$ and $f$ are the only edges with colour $x$, such an automorphism must fix the vertex at which $e$ and $f$ meet. Thus it has a fixed point and by Lemma~\ref{lem:nofixedpoint} it must be the identity.

So assume that no two incident edges receive the same colour in $c'$. Take an arbitrary edge $e$ and colour it with colour $k$. Then recolour an edge $f$ incident to $e$ with colour $c'(e)$. An automorphism which preserves the resulting colouring must fix the edge $e$ because it is the only edge with colour $x$. It must also fix the edge $f$ because it is the only edge incident to $e$ which has colour $c'(e)$. Since all the other edges have the same colours as in $c'$, the automorphism in question must preserve $c'$. But it also has to fix the vertex where $e$ and $f$ meet and thus by Lemma~\ref{lem:nofixedpoint} it is the identity.
\end{proof}

\begin{cor}
\label{cor:infinite_equal}
If $D(G)$ is infinite then $D'(G) \leq D(G)$.
\end{cor}
\begin{proof}
For an infinite cardinal $\alpha$ we have $1+ \alpha = \alpha$.
\end{proof}

\bibliographystyle{abbrv}
\bibliography{ref}

\end{document}